\documentclass[12pt]{amsart}

\usepackage[letterpaper,margin=1.1in]{geometry}
\usepackage{amssymb} 
\usepackage{graphicx}
\usepackage{xcolor}

\newtheorem{theorem}{Theorem}[section]
\newtheorem{lemma}[theorem]{Lemma}
\newtheorem{proposition}[theorem]{Proposition}

\theoremstyle{definition}
\newtheorem{definition}[theorem]{Definition}

\newtheorem{problem}[theorem]{Problem}

\theoremstyle{remark}
\newtheorem{remark}[theorem]{Remark}
\newtheorem{example}[theorem]{Example}
\newtheorem*{remark-no-number}{Remark}

\newcommand{\RR}{\mathbb{R}}

\newcommand{\XX}{\mathbb{X}}
\newcommand{\Haus}{\mathcal{H}}

\newcommand{\tJ}{\widetilde{J}}

\newcommand{\dist}{\mathop\mathrm{dist}\nolimits}
\newcommand{\diam}{\mathop\mathrm{diam}\nolimits}
\newcommand {\side}{\mathop\mathrm{side}\nolimits}
\newcommand{\lD}[1]{\mathop{\underline{D}^{#1}}\nolimits}
\newcommand{\uD}[1]{\mathop{\overline{D}^{\,#1}}\nolimits}
\newcommand{\spt}{\mathop\mathrm{spt}\nolimits}

\newcommand{\res}{\hbox{ {\vrule height .22cm}{\leaders\hrule\hskip.2cm} }} 

\newcommand{\rect}{{rect}}
\newcommand{\pu}{{pu}}

\usepackage{mathrsfs}

\numberwithin{figure}{section}
\numberwithin{equation}{section}

\begin{document}

\title[Generalized rectifiability of measures]{Generalized rectifiability of measures and\\ the identification problem}
\thanks{The author was partially supported by NSF DMS grants 1500382 and 1650546.}
\date{January 20, 2019}
\author{Matthew Badger}
\address{Department of Mathematics\\ University of Connecticut\\ Storrs, CT 06269-3009}
\email{matthew.badger@uconn.edu}

\begin{abstract}One goal of geometric measure theory is to understand how measures in the plane or a higher dimensional Euclidean space  interact with families of lower dimensional sets. An important dichotomy arises between the class of rectifiable measures, which give full measure to a countable union of the lower dimensional sets, and the class of purely unrectifiable measures, which assign measure zero to each distinguished set. There are several commonly used definitions of rectifiable and purely unrectifiable measures in the literature (using different families of lower dimensional sets such as Lipschitz images of subspaces or Lipschitz graphs), but all of them can be encoded using the same framework. In this paper, we describe a framework for generalized rectifiability, review a selection of classical results on rectifiable measures in this context, and survey recent advances on the identification problem for Radon measures that are carried by Lipschitz or H\"older or $C^{1,\alpha}$ images of Euclidean subspaces, including theorems of Azzam-Tolsa, Badger-Schul, Badger-Vellis, Edelen-Naber-Valtorta, Ghinassi, and Tolsa-Toro.

This survey paper is based on a talk at the Northeast Analysis Network Conference held in Syracuse, New York in September 2017.
\end{abstract}

\maketitle

\setcounter{tocdepth}{1}
\tableofcontents

\section{Introduction}

Given a measure, perhaps one of the most fundamental problems is to determine which sets have positive measure and which sets have zero measure. In this paper, we are interested in a \emph{dual problem}: given a class of sets, we want to determine which measures assign all of their mass to those sets and which measures vanish on each of those sets. Special cases of the dual problem are commonly studied in geometric measure theory, under the heading of rectifiability of measures. To formally state this (see Problem \ref{prob:ident}), we need to first introduce some terminology (see Definition \ref{def:carry}), which seems to be missing from the standard lexicon.
Recall that a \emph{measurable space} $(\mathbb{X},\mathcal{M})$ is a nonempty set $\mathbb{X}$, equipped with a $\sigma$-algebra $\mathcal{M}$, i.e.~a nonempty family of subsets of $\XX$ that is closed under taking complements and countable unions. By \emph{measure}, we mean a positive measure, i.e.~a function $\mu:\mathcal{M}\rightarrow[0,\infty]$ with $\mu(\emptyset)=0$ that is countably additive on disjoint sets.

\begin{definition}\label{def:carry} Let $(\mathbb{X},\mathcal{M})$ be a measurable space, let $\mu$ be a measure defined on $\mathcal{M}$, and let $\mathcal{N}\subseteq\mathcal{M}$ be a family of measurable sets. We say that \begin{enumerate}
\item \emph{$\mu$ is carried by $\mathcal{N}$} if there exist countably many $N_i\in\mathcal{N}$ such that $\mu(\mathbb{X}\setminus\bigcup_i N_i)=0$;
\item \emph{$\mu$ is singular to $\mathcal{N}$} if $\mu(N)=0$ for every $N\in\mathcal{N}$.
\end{enumerate}
\end{definition}

In a measurable metric space $(\mathbb{X},\mathcal{M})$ such that  $\mathcal{M}$ contains the Borel $\sigma$-algebra $\mathcal{B}_\XX$, the \emph{support} of a measure $\mu$ defined on $\mathcal{M}$ is the closed set defined by $$\spt\mu:=\{x\in\XX:\mu(B(x,r))>0\text{ for all }r>0\}.$$ Equivalently, the support of $\mu$ is the smallest closed set $F\subseteq\XX$ such that $\mu(\XX\setminus F)=0$.

\begin{figure}\includegraphics[width=.8\textwidth]{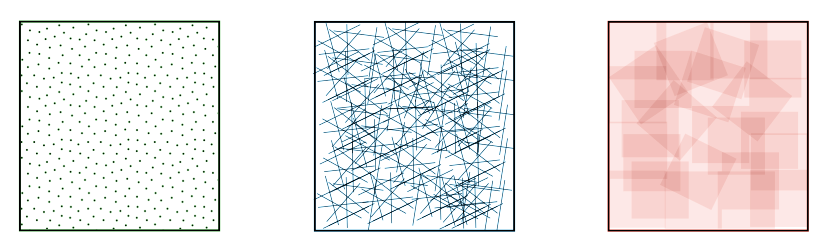}
\caption{Three measures in $\RR^2$ carried by points, line segments, and squares} \label{fig:3m}\end{figure}

\begin{example}\label{ex:3} Let $\{x_i\}_1^\infty$, $\{I_i\}_1^\infty$, $\{S_i\}_{1}^\infty$ be dense sequences of points, unit line segements, and unit squares in $\RR^2$ (see Figure \ref{fig:3m}), meaning that $\overline{\bigcup_{i=1}^\infty \{x_i\}}=\overline{\bigcup_{i=1}^\infty I_i}=\overline{\bigcup_{i=1}^\infty S_i}=\RR^2,$ and let $a_i>0$ be a sequence of weights such that $\sum_1^\infty a_i=1$. Define three Borel measures $\mu_0$, $\mu_1$, and $\mu_2$  on $\RR^2$ by $$\mu_0:=\sum_{1}^\infty a_i \delta_{x_i},\qquad \mu_1:=\sum_1^\infty a_i \mathscr{L}^1\res I_i,\qquad \mu_2:=\sum_1^\infty a_i \mathscr{L}^2\res S_i,$$ where $\delta_{x_i}$ is the Dirac measure at $x_i$, $\mathscr{L}^1\res I_i$ is 1-dimensional Lebesgue measure restricted to $I_i$, and $\mathscr{L}^2\res S_i$ is 2-dimensional Lebesgue measure restricted to  $S_i$. The measures $\mu_0$, $\mu_1$, and $\mu_2$ are probability measures on $\RR^2$ with common support, $$\spt\mu_0=\spt\mu_1=\spt\mu_2=\RR^2.$$ However, $\mu_0$ is carried by $0$-dimensional sets (points), $\mu_1$ is carried by $1$-dimensional sets (line segements), and $\mu_2$ is carried by $2$-dimensional sets (squares). On the other hand, $\mu_1$ and $\mu_2$ are singular to points and $\mu_2$ is singular to lines. Thus, the three measures are distinguished by their underlying carrying sets. This shows that \emph{the support of a measure is a rough approximation that may hide underlying structure of a measure}. Definition \ref{def:carry} provides us with language to discuss this structure.
\end{example}

A validation of Definition \ref{def:carry} is that every $\sigma$-finite measure can be uniquely written as the sum of a measure carried by $\mathcal{N}$ and a measure singular to $\mathcal{N}$. In the statement of Proposition \ref{p:decomp} and below, we let $\mu\res A$ denote the \emph{restriction} of a measure $\mu$ to a measurable set $A$; i.e., $\mu\res A$ is the measure defined by the rule $(\mu\res A)(B)=\mu(A\cap B)$ for all measurable sets $B$. A measure is \emph{$\sigma$-finite} if $\mathbb{X}=\bigcup_{i=1}^\infty X_i$ for some sets $X_i\in \mathcal{M}$ with $\mu(X_i)<\infty$ for all $i$.

\begin{proposition}[decomposition of $\sigma$-finite measures] \label{p:decomp} Let $(\mathbb{X},\mathcal{M})$ be a measurable space and let $\mathcal{N}\subseteq\mathcal{M}$ be a family of measurable sets. If $\mu$ is a $\sigma$-finite measure defined on $\mathcal{M}$, then $\mu$ can be uniquely written as \begin{equation}\mu=\mu_\mathcal{N}+\mu_\mathcal{N}^\perp,\end{equation} where $\mu_\mathcal{N}$ and $\mu_{\mathcal{N}}^\perp$ are measures defined on $\mathcal{M}$ such that $\mu_\mathcal{N}$ is carried by $\mathcal{N}$ and $\mu_\mathcal{N}^\perp$ is singular to $\mathcal{N}$. Moreover, there exists $A_\mu\in\mathcal{M}$ (a countable union of sets in $\mathcal{N}$) such that \begin{equation}\mu_\mathcal{N}=\mu\res A_\mu,\qquad \mu_\mathcal{N}^\perp=\mu\res (\mathbb{X}\setminus A_\mu).\end{equation} The measure $\mu$ is carried by $\mathcal{N}$ if and only if $\mu=\mu_\mathcal{N}$; and, similarly, $\mu$ is singular to $\mathcal{N}$ if and only if $\mu=\mu_{\mathcal{N}}^\perp$.\end{proposition}

\begin{proof} This is an elementary exercise in measure theory. We present a portion of the proof in order to motivate the identification problem (see Problem \ref{prob:ident}). Assume that $\mathbb{X}=\bigcup_{i=1}^\infty X_i$ for some sequence $X_i\in\mathcal{M}$ of measurable sets with $\mu(X_i)<\infty$. Let $\widetilde{\mathcal{N}}$ denote the collection of finite unions of sets in $\mathcal{N}$, and for each $j\geq 1$, define $$\alpha_j:=\sup_{A\in\widetilde{\mathcal{N}}}\mu((X_1\cup\dots\cup X_j)\cap A) \leq \sum_{i=1}^j\mu(X_i)<\infty.$$ By the approximation property of supremum, for each $j\geq 1$, we can choose a set $A_j\in\widetilde{\mathcal{N}}$ such that $\mu((X_1\cup\dots\cup X_j)\cap A_j) \geq \alpha_j - 1/j$. We now define $$\mu_{\mathcal{N}}:=\mu\res \left(\bigcup_{j=1}^\infty A_j\right),\qquad \mu_{\mathcal{N}}^\perp:=\mu\res\left(\mathbb{X}\setminus \bigcup_{j=1}^\infty A_j\right).$$ Clearly $\mu=\mu_{\mathcal{N}}+\mu_{\mathcal{N}}^\perp$ and the measure $\mu_\mathcal{N}$ is carried by $\mathcal{N}$. We leave it to the reader to verify the measure $\mu_\mathcal{N}^\perp$ is singular to $\mathcal{N}$ and that the decomposition of $\mu$ as the sum of a measure carried by $\mathcal{N}$ and a measure singular to $\mathcal{N}$ is unique. A complete proof can be found in the appendix of \cite{BV}.\end{proof}

\begin{example}Let $(\mathbb{X},\mathcal{M})$ be a measurable space. Let $\nu$ be a measure defined on $\mathcal{M}$ and consider $\mathcal{N}:=\{A\in\mathcal{M}: \nu(A)=0\}$. Applying Proposition \ref{p:decomp}, we recover the familiar fact that if $\mu$ is a $\sigma$-finite measure defined on $\mathcal{M}$, then $$\mu=\sigma+\rho,$$ where $\sigma:=\mu_{\mathcal{N}}$ and $\rho:=\mu_\mathcal{N}^\perp$ satisfy $\sigma\perp \nu$ and $\rho\ll\nu$. In other words, we obtain the Lebesgue-Radon-Nikodym decomposition of a $\sigma$-finite measure $\mu$ as the sum of a singular measure and an absolutely continuous measure with respect to an arbitrary auxiliary measure $\nu$.\end{example}

\begin{example}Let $(\mathbb{X},\mathcal{M})$ be a measurable space. Let $\mathcal{N}$ denote the collection of \emph{atoms} of $\mathcal{M}$; that is, $A\in\mathcal{N}$ if and only if $A\in\mathcal{M}$ is nonempty and $B\subsetneq A$ for some $B\in\mathcal{M}$ implies $B=\emptyset$. (In the common situtation that $\mathbb{X}$ is a metric space and $\mathcal{M}$ contains the Borel $\sigma$-algebra on $\mathbb{X}$, the family of atoms of $\mathcal{M}$ is precisely the set of singletons of $\mathbb{X}$.) For every atom $A$, the \emph{Dirac measure} $\delta_A$ is defined by the rule $\delta_A(B)=1$, if $A\subseteq B$, and $\delta_A(B)=0$, otherwise. By Proposition \ref{p:decomp}, we obtain that for any $\sigma$-finite measure $\mu$ defined on $\mathcal{M}$, $$\mu=\mu_{\bullet}+\mu_{\circ},$$ where $\mu_{\bullet}:=\mu_\mathcal{N}$ is \emph{atomic} (carried by the atoms of $\mathcal{M}$) and $\mu_{\circ}:=\mu_\mathcal{N}^\perp$ is \emph{atomless} (singular to the atoms of $\mathcal{M}$). Moreover, in this situation, one can easily verify that either $\mu_\bullet=0$ or there exist finitely or countably many atoms $A_i$ and weights $0<c_i<\infty$ such that $$\mu_\bullet=\sum_{i}c_i \delta_{A_i}.$$\end{example}

\begin{example}\label{ex:1rect} Assume that $(\mathbb{X},\mathcal{M})=(\RR^n,\mathcal{B}_{\RR^n})$, where $\mathcal{B}_{\RR^n}$ denotes the Borel $\sigma$-algebra on $\RR^n$. A \emph{rectifiable curve} in $\RR^n$ is the image $f([0,1])$ of a continuous map $f:[0,1]\rightarrow \RR^n$ of finite total variation, or equivalently, is the image of a Lipschitz continuous map $f:[0,1]\rightarrow\RR^n$ (see e.g.~\cite{AO}). Let $\mathcal{N}$ denote the collection of all rectifiable curves in $\RR^n$. By Proposition \ref{p:decomp}, every $\sigma$-finite Borel measure $\mu$ on $\RR^n$ can be uniquely written as $$\mu=\mu^1_\rect+\mu^1_\pu,$$ where $\mu^1_\rect:=\mu_\mathcal{N}$ is \emph{1-rectifiable} (carried by rectifiable curves) and $\mu^1_\pu:=\mu_\mathcal{N}^\perp$ is \emph{purely 1-unrectifiable} (singular to rectifiable curves). For instance, of the three measures in Example \ref{ex:3}, the measures $\mu_0$ and $\mu_1$ are 1-rectifiable, while the measure $\mu_2$ is purely 1-unrectifiable. Finding geometric and/or measure-theoretic properties that distinguish between 1-rectifiable and purely 1-unrectifiable measures and their higher dimensional variants constitutes a large program in geometric measure theory; see e.g. \cite{Federer}, \cite{Falconer}, \cite{Mattila}. We will define and discuss $m$-rectifiable measures and purely $m$-unrectifiable measures with $m>1$ in \S3. \end{example}

Although Proposition \ref{p:decomp} provides for the decomposition of any $\sigma$-finite measure into component measures carried by or singular to $\mathcal{N}$, the proof of this fact is abstract (as it relies on the completeness axiom of $\RR$ and the approximation property of the supremum) and does not provide a concrete method to identify the components for a particular measure. This leads us to the following problem, which is our main problem of interest.

\begin{problem}[identification problem] \label{prob:ident} Let $(\mathbb{X},\mathcal{M})$ be a measurable space, let $\mathcal{N}\subseteq\mathcal{M}$ be a family of measurable sets, and let $\mathscr{F}$ be a family of $\sigma$-finite measures defined on $\mathcal{M}$. Find properties $P(\mu,x)$ and $Q(\mu,x)$ defined for all $\mu\in\mathscr{F}$ and $x\in\mathbb{X}$ such that \begin{equation}\mu_\mathcal{N} = \mu\res\{x\in\mathbb{X}: P(\mu,x) \text{ holds}\}\quad\text{and}\quad\mu_\mathcal{N}^\perp = \mu\res\{x\in\mathbb{X}: Q(\mu,x) \text{ holds}\}.\end{equation} In other words, find (pointwise) properties that identify the part of $\mu$ carried by $\mathcal{N}$ and the part of $\mu$ singular to $\mathcal{N}$ for all measures in the class $\mathscr{F}$.
\end{problem}

There is room for debate on what constitutes a ``good" solution of Problem \ref{prob:ident}, but in the author's view a reasonable solution should generally involve the geometry of the space $\mathbb{X}$ or sets $\mathcal{N}$. If this includes the ability to sample a measure on a ball, then the atomic identification problem for locally finite measures in a metric space is easily solved.

\begin{example}Let $(\mathbb{X},\mathcal{M})=(\mathbb{X},\mathcal{B}_{\mathbb{X}})$, where $\mathbb{X}$ is a metric space and $\mathcal{B}_\mathbb{X}$ denotes the Borel $\sigma$-algebra on $\mathbb{X}$. Let $\mathcal{N}$ be the collection of singletons in $\mathbb{X}$ and let $\mathscr{F}$ denote the collection of locally finite, $\sigma$-finite Borel measures. If $\mu\in\mathscr{F}$, then $$\mu_\bullet\equiv \mu_{\mathcal{N}}=\mu\res\left\{x\in \mathbb{X}: \lim_{r\downarrow 0} \mu(B(x,r))>0\right\},$$
$$\mu_\circ\equiv\mu_\mathcal{N}^\perp=\mu\res\left\{x\in \mathbb{X}: \lim_{r\downarrow 0}\mu(B(x,r))=0\right\}$$ by continuity of measures from above.

Here the restriction to locally finite measures is crucial. For example, if $\{\ell_i\}_{i=1}^\infty$ is an enumeration of straight lines in the plane that pass through the origin and have rational slopes, then $\mu=\sum_{i=1}^\infty \mathscr{L}^1\res \ell_i$ is $\sigma$-finite and atomless, but $\mu(B(0,r))=\infty$ for all $r>0$.\end{example}

The identification problem for 1-rectifiable measures was first studied by Besicovitch \cite{Bes28,Bes38} in a broader investigation into the geometry of planar sets with positive and finite length, and later by Morse and Randolph \cite{MR}, Moore \cite{Moore}, Pajot \cite{Pajot97}, Lerman \cite{Lerman}, and Azzam and Mourgoglou \cite{AM15}. Complete solutions within the classes of pointwise doubling measures and Radon measures in $\RR^n$ were furnished very recently by Badger and Schul \cite{BS3}. A description of the latter will be presented in \S2.

\begin{example}\label{ex:bes} Let $(\mathbb{X},\mathcal{M})=(\RR^2,\mathcal{B}_{\RR^2})$, let $\mathcal{N}$ be the collection of rectifiable curves in $\RR^2$, and let $\mathscr{F}:=\{\Haus^1\res E:E\in\mathcal{B}_{\RR^2}\text{ and } 0<\Haus^1(E)<\infty\}$, where $\Haus^s$ denotes the $s$-dimensional Hausdorff measure defined by \begin{equation}\label{e:haus}\Haus^s(E):=\lim_{\delta\downarrow 0} \inf\left\{\sum_{i=1}^\infty (\diam E_i)^s: E\subseteq\bigcup_{i=1}^\infty E_i\text{ and }\diam E_i\leq \delta\right\}.\end{equation} Besicovitch \cite{Bes28,Bes38} proved that if $\mu=\Haus^1\res E\in\mathscr{F}$, then $$\mu^1_\rect\equiv\mu_\mathcal{N}=\Haus^1\res\left\{x\in E: \lim_{r\downarrow 0} \frac{\Haus^1(E\cap B(x,r))}{2r}=1\right\},$$
$$\mu^1_\pu\equiv\mu_\mathcal{N}^\perp=\Haus^1\res\left\{x\in E: \liminf_{r\downarrow 0} \frac{\Haus^1(E\cap B(x,r))}{2r}\leq\frac{3}{4}\right\}.$$ Besicovitch conjectured that the constant 3/4 may be replaced by 1/2, but finding the optimal constant is still an open problem. The best result to date is due to Preiss and Ti\v{s}er \cite{PT92}, who showed that 3/4 may be replaced by the constant $(2+\sqrt{46})/12=0.731\dots$; moreover, Preiss and Ti\v{s}er showed this holds in any metric space. A special case of the 1/2 conjecture for sets with an \emph{a priori} flatness condition was settled by Farag \cite{Farag1,Farag2}.

In \cite{Bes28}, Besicovitch also showed that if $\mu=\Haus^1\res E\in\mathscr{F}$, then $$\mu^1_\rect=\Haus^1\res\left\{x\in E:E\text{ has an $\Haus^1$ approximate tangent line at }x\right\},$$
$$\mu^1_\pu=\Haus^1\res\left\{x\in E: E\text{ does not have an $\Haus^1$ approximate tangent line at }x\right\}.$$ Here we say that $E$ has an \emph{$\Haus^1$ approximate tangent line} at $x$ if $$\limsup_{r\downarrow 0}\frac{\Haus^1(E\cap B(x,r))}{2r}>0$$ and there exists a line $L$ containing $x$ such that $$\limsup_{r\downarrow 0} \frac{\Haus^1(E\cap B(x,r)\setminus X(x,L,\alpha))}{2r}=0\quad\text{for all }\alpha\in(0,\pi/2],$$ where $X(x,L,\alpha)$ denotes the cone of points $y\in\RR^2$ such that the line $\{x+t(y-x):t\in\RR\}$ meets the line $L$ at an angle at most $\alpha$ (see Figure 1.2). For a contemporary, self-contained presentation of these results, see \cite{Falconer}. \end{example}

\begin{figure}\begin{center}\includegraphics[width=.7\textwidth]{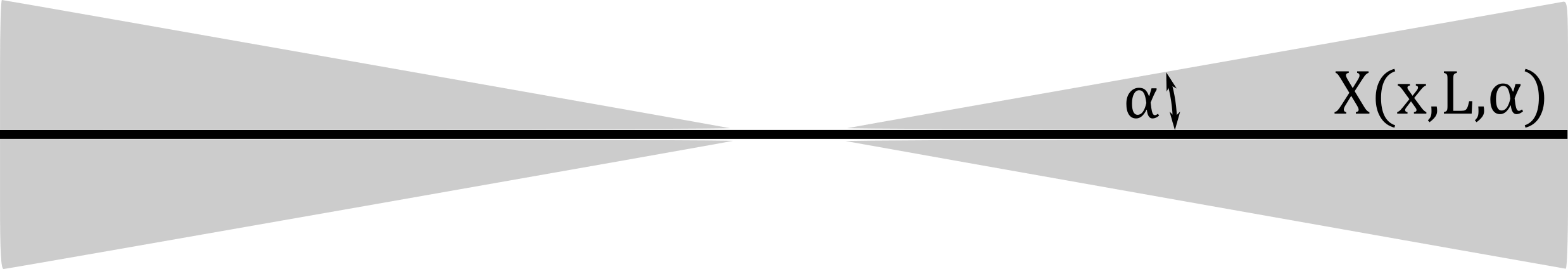}\end{center}\caption{the cone $X(x,L,\alpha)$}\end{figure}

\begin{example}Let $(\mathbb{X},\mathcal{M})=(\RR^2,\mathcal{B}_{\RR^2})$, let $\mathcal{N}$ be the collection of rectifiable curves in $\RR^2$, and let $\mathscr{F}$ denote the collection of Radon measures (locally finite, Borel regular measures) $\mu$ on $\RR^2$ \textbf{such that} $\mu\ll\Haus^1$. Morse and Randolph \cite{MR} proved that if $\mu\in\mathscr{F}$, then $$\mu^1_\rect\equiv\mu_\mathcal{N}=\mu\res\left\{x\in\RR^2: \lim_{r\downarrow 0}\frac{\mu(B(x,r))}{2r}\text{ exists, and }0<\lim_{r\downarrow 0} \frac{\mu(B(x,r))}{2r}<\infty\right\},$$
$$\mu^1_\pu\equiv\mu_\mathcal{N}^\perp=\mu\res\left\{x\in\RR^2: \liminf_{r\downarrow 0} \frac{\mu(B(x,r))}{2r}\leq \frac{100}{101}\limsup_{r\downarrow 0} \frac{\mu(B(x,r))}{2r}\right\}.$$ Moore \cite{Moore} extended Morse and Randolph's density characterizations of $\mu^1_\rect$ and $\mu^1_\pu$ to Radon measures $\mu$ in $\RR^n$ such that $\mu\ll\Haus^1$ for all dimensions $n\geq 2$. Higher dimensional analogues of the density characterization of absolutely continuous $m$-rectifiable measures in $\RR^n$ were established by Preiss \cite{Preiss}.

In \cite{MR}, Morse and Randolph also proved that if $\mu$ is a Radon measure in $\RR^2$ such that $\mu\ll\Haus^1$, then $$\mu^1_\rect=\mu\res\left\{x\in\RR^n:\text{there is a $\mu$ approximate tangent line at }x\right\},$$
$$\mu^1_\pu=\mu\res\left\{x\in \RR^n: \text{ there is not a $\mu$ approximate tangent line at }x\right\},$$ where \emph{$\mu$ approximate tangent lines} are defined by analogy with $\Haus^1\res E$ approximate tangent lines in an obvious way (replacing $\Haus^1\res E$ with $\mu$). Higher dimensional analogues of the tangential characterization of absolutely continuous $m$-rectifiable measures in $\RR^n$ were supplied by Federer \cite{Fed47}.
\end{example}

\begin{example}\label{ex:beta} For any Radon measure $\mu$ on $\RR^n$, location $x\in\RR^n$, and scale $r>0$, define the \emph{homogeneous 1-dimensional $L^2$ Jones beta number} $\beta^h_2(\mu,x,r)$ by \begin{equation}\label{e:beta-h}\beta^h_2(\mu,x,r)^2=\inf_{\ell}\int_{B(x,r)} \left(\frac{\dist(y,\ell)}{r}\right)^2 \frac{d\mu(y)}{r},\end{equation} where the infimum runs over all straight lines (1-dimensional affine subspaces) in $\RR^n$. Roughly speaking, $\beta^h_2(\mu,x,r)$ records in an $L^2$ gauge how well the measure $\mu\res B(x,r)$ can be approximated by a measure supported on a line. Here \emph{homogeneous} refers to the normalization of the measure in \eqref{e:beta-h}---see \S2 below for further discussion. The number $\beta^h_2(\mu,x,r)$ is one instance of a larger family of unilateral approximation numbers that originated in work of Jones \cite{Jones-TSP}, Okikolu \cite{Ok-TST}, Bishop and Jones \cite{BJ}, and David and Semmes \cite{DS91,DS93} on quantitative rectifiability of sets. The \emph{$L^2$ Jones function} or \emph{geometric square function} defined by \begin{equation}\label{e:Jones-h} J_2(\mu,x):=\int_0^1 \beta^h_2(\mu,x,r)^2\, \frac{dr}{r}\end{equation} records the total square error of the linear approximation numbers $\beta^h_2(\mu,x,r)$ on vanishing coarse scales.

Let $(\mathbb{X},\mathcal{M})=(\RR^n,\mathcal{B}_{\RR^n})$, let $\mathcal{N}$ be the collection of rectifiable curves in $\RR^n$, and let $\mathscr{F}:=\{\Haus^1\res E:E\text{ is compact and } \Haus^1(E\cap B(x,r))\sim_E r\text{ for all }0<r\leq \diam E\}$ denote the collection of  Hausdorff measures restricted to \textbf{compact 1-dimensional Ahlfors regular sets}. Pajot \cite{Pajot97} proved that if $\mu=\Haus^1\res E\in\mathscr{F}$, then $$\mu^1_\rect\equiv\mu_\mathcal{N}=\Haus^1\res\left\{x\in E: J_2(\mu,x)<\infty\right\},$$
$$\mu^1_\pu\equiv\mu_\mathcal{N}^\perp=\Haus^1\res\left\{x\in E: J_2(\mu,x)=\infty\right\}.$$ Pajot also obtained analogous results for $m$-dimensional Hausdorff measures restricted to compact $m$-dimensional Ahlfors regular sets using higher-dimensional beta numbers, which are defined using approximation by $m$-planes instead of approximation by lines.

Although the definition of $J_2(\mu,x)$ makes sense for any Radon measure, it  does not allow one to identify 1-rectifiable measures in general.
In \cite{Lerman}, Lerman explored modifications to the definition of $\beta^h_2(\mu,x,r)$ and $J_2(\mu,x)$ that expanded the class of measures to which Pajot's theorem applies. Further work in this direction (see \cite{BS,BS2,BS3}) led to a solution of the identification problem for 1-rectifiable Radon measures, which we will describe in detail in the next section. In the special case when $\mu$ is a \emph{pointwise doubling} Radon measure (i.e.~$\limsup_{r\downarrow 0} \mu(B(x,2r))/\mu(B(x,r))<\infty$ at $\mu$-a.e.~$x\in\RR^n$), the method of \cite{BS3} can be used to show that the \emph{density-normalized $L^2$ Jones function} $$\widetilde J_2(\mu,x) := \int_0^1 \beta^h_2(\mu,x,r)^2\cdot\left[\frac{r}{\mu(B(x,r))}\right]^2\,\frac{dr}{r}$$ determines the 1-rectifiable and purely 1-unrectifiable part of $\mu$: $$\mu^1_\rect\equiv\mu_\mathcal{N}=\mu\res\left\{x\in \RR^n: \widetilde J_2(\mu,x)<\infty\right\},$$
$$\mu^1_\pu\equiv\mu_\mathcal{N}^\perp=\mu\res\left\{x\in \RR^n: \widetilde J_2(\mu,x)=\infty\right\}.$$ For a 1-dimensional Ahlfors regular measure, it is clear that $J_2(\mu,x)\sim \widetilde J_2(\mu,x)$. Therefore, Badger and Schul's pointwise doubling theorem directly generalizes the case $m=1$ of Pajot's theorem to a wider class of measures, including measures like $\mu_0$ from Example \ref{ex:3} that are mutually singular with $\Haus^1$. Martikainen and Orponen \cite{MO} have constructed a family of pointwise non-doubling Borel probability measures $\mu_\varepsilon$ on $\RR^2$ such that $$\tJ_2(\mu_\varepsilon,x)\leq \varepsilon\ll 1\quad\text{for all }x\in\RR^2$$ and $\mu_\varepsilon$ is purely 1-unrectifiable. Thus, for general Radon measures, pointwise control on the density-normalized Jones function $\tJ_2(\mu,x)$ is also not enough to identify $\mu^1_\rect$.
\end{example}

\begin{example}Let $(\mathbb{X},\mathcal{M})=(\mathbb{X},\mathcal{B}_\mathbb{X})$ be a \textbf{connected} metric space equipped with its Borel $\sigma$-algebra, let $\mathcal{N}$ denote the collection of rectifiable curves in $\mathbb{X}$, and let $\mathscr{F}$ denote the collection of Radon measures on $\mathbb{X}$ whose supports are $\mathbb{X}$ (i.e.~ $\mu(B)>0$ for every ball $B$ in $\mathbb{X}$) and are doubling in the sense that $$\sup\left\{\frac{\mu(B(x,2r))}{\mu(B(x,r))}:x\in\mathbb{X}, r>0\right\}<\infty.$$ Azzam and Mourgoglou \cite{AM15} proved that $$\mu^1_\rect\equiv\mu_\mathcal{N}=\mu\res\left\{x\in\mathbb{X}: \liminf_{r\downarrow 0}\frac{\mu(B(x,r))}{2r}>0\right\},$$ $$\mu^1_\pu\equiv\mu_{\mathcal{N}}^\perp=\mu\res\left\{x\in\mathbb{X}: \liminf_{r\downarrow 0}\frac{\mu(B(x,r))}{2r}=0\right\}.$$  Standard examples (self-similar Cantor sets in $\RR^2$ of Hausdorff dimension one) show that the connectedness hypothesis cannot be dropped from Azzam and Mourgolgou's theorem. For examples in $\RR^n$ of 1-rectifiable doubling measures whose support is the whole of $\RR^n$, see Garnett, Killip, and Schul \cite{GKS}.\end{example}

The examples above illustrate different solutions of Problem \ref{prob:ident} when $\mathcal{N}$ is the collection of rectifiable curves in $\RR^n$ and $\mathscr{F}$ is one of several sets of $\sigma$-finite Borel measures on $\RR^n$. Additional results are available when $\mathcal{N}$ is the collection of images of Lipschitz maps $f:[0,1]^m\rightarrow\RR^n$ \textbf{and} $\mu\ll\Haus^m$ (see \S3), or when $\XX=\mathbb{H}^n$ is the $n$-th Heisenberg group (see \S4), but in general we currently know far less than one should like. For example, even the following deceptively simple problem is presently open (cf. Example \ref{ex:3}).

\begin{problem}\label{prob:lines} Let $\mathcal{N}$ denote the set of lines (1-dimensional affine subspaces) in $\RR^2$. Identify the Radon measures on $\RR^2$ that are carried by $\mathcal{N}$ or singular to $\mathcal{N}$.\end{problem}

The rest of this survey is organized, as follows. In \S2, we present the solution of the identification problem for 1-rectifiable Radon measures from \cite{BS3}. In \S3, we review the current state of affairs on identification problems for $m$-rectifiable measures when $m\geq 2$.
In \S4, we discuss further directions and open problems on generalized rectifiability, including fractional rectifiability and higher order rectifiability in $\RR^n$---and other spaces.


\section{Solution of the identification problem for 1-rectifiable measures}

In this section, we will present the solution to the identification problem for 1-rectifiable Radon measures from \cite{BS3} (see Theorem \ref{t:big}). The basic strategy is to promote a characterization of subsets of rectifiable curves, called the Analyst's Traveling Salesman theorem, to a characterization of 1-rectifiable measures. For any $E\subset \RR^n$ and bounded set $Q\subset\RR^n$ of positive diameter (such as a ball or a cube), the \emph{Jones' beta number} $\beta_E(Q)$ is the quantity in $[0,1]$ defined by $$\beta_E(Q):=\inf_\ell \sup_{x\in E\cap Q} \frac{\dist(x,\ell)}{\diam Q},$$ where $\ell$ ranges over all straight lines in $\RR^n$, if $E\cap Q\neq\emptyset$, and by $\beta_E(Q):=0$ if $E\cap Q=\emptyset$. At one extreme, if $\beta_E(Q)=0$, then $E\cap Q$ is a subset of some straight line $\ell$ passing through $Q$. At the other extreme, if $\beta_E(Q)\sim 1$, then $E\cap Q$ is ``far away" from being contained in any line passing through $Q$. The following theorem was first conceived and proved with by Jones \cite{Jones-TSP} when $n=2$ and later extended to higher dimensions $n\geq 3$ by Okikiolu \cite{Ok-TST}. For further information, see the survey \cite{TST-survey} by Schul as well as the subsequent developments in the Heisenberg group by Li and Schul \cite{LS1,LS2} and in Laakso type spaces by David and Schul \cite{DS-tst}.

\begin{theorem}[Analyst's Traveling Salesman theorem] \label{t:tst} Let $n\geq 2$ and let $E\subset\RR^n$ be a bounded set. Then $E$ is contained in a rectifiable curve in $\RR^n$ if and only if $$S_E:=\sum_{Q\text{ dyadic}} \beta_E(3Q)^2\diam Q <\infty,$$ where the sum ranges over all dyadic cubes and $3Q$ denotes the concentric dilate of $Q$. Moreover, if $S_E<\infty$, there exists a rectifiable curve $\Gamma$ containing $E$ such that $$\Haus^1(\Gamma)\lesssim_n \diam E+ S_E.$$\end{theorem}

The characterization of 1-rectifiable measures from \cite{BS3} has two main components: \begin{itemize}
\item the lower 1-dimensional Hausdorff density $\lD1(\mu,\cdot)$ of a measure $\mu$ (a common notion in geometric measure theory); and,
\item a density-normalized Jones function $J^*_2(\mu,\cdot)$ associated to ``anisotropic" $L^2$ beta numbers $\beta^*_2(\mu,\cdot)$ (which are the main innovation in \cite{BS3}).
\end{itemize} (Below I attempt to provide an intuitive explanation of some ideas behind the definition of $\beta^*_2(\mu,x)$ and $J_2^*(\mu,x)$ that R. Schul and I had in mind when we wrote, but did not expressly describe in \cite{BS3}.)

For any Radon measure $\mu$ on $\RR^n$ and $x\in\RR^n$, the \emph{lower 1-dimensional Hausdorff density} of $\mu$ at $x$ is the quantity $\lD1(\mu,x)$ in $[0,\infty]$ given by $$\lD1(\mu,x):=\liminf_{r\downarrow 0} \frac{\mu(B(x,r))}{2r}.$$ The following lemma shows that to identify the 1-rectifiable part of $\mu$, it suffices to focus on the set of points where the lower density is positive.

\begin{lemma}[see {\cite[Lemma 2.7]{BS}}] \label{l:larp} Let $\mu$ be a Radon measure on $\RR^n$. Then $$\mu^1_\rect \leq \mu\res\{x: \lD1(\mu,x)>0\}\quad\text{ and }\quad \mu\res\{x:\lD1(\mu,x)=0\}\leq \mu^1_\pu.$$ That is, $\mu\res\{x:\lD1(\mu,x)=0\}$ is purely 1-unrectifiable, and if $\mu$ is 1-rectifiable, then $\lD1(\mu,x)>0$ at $\mu$-a.e. $x\in\RR^n$. \end{lemma}

\begin{proof} This is an easy consequence of the relationship between pointwise control on the lower density $\lD1(\mu,\cdot)$ along a set $E$ and the 1-dimensional packing measure $\mathcal{P}^1$ of $E$ (see Taylor and Tricot \cite{TT}); the finiteness of the packing measure $\mathcal{P}^1$ on bounded sets in $\RR$; and, the interaction of packing measures with Lipschitz maps. See \cite{BS} for details.\end{proof}

For any Radon measure $\mu$ on $\RR^n$, bounded Borel set $Q\subset\RR^n$ with positive diameter (typically we take $Q$ to be a ball or a cube), and straight line $\ell$ on $\RR^n$, we define the quantity $\beta_2(\mu,Q,\ell)$ by $$\beta_2(\mu,Q,\ell)^2=\int_Q \left(\frac{\dist(y,\ell)}{\diam Q}\right)^2 \frac{d\mu(y)}{\mu(Q)}.$$ The \emph{non-homogeneous 1-dimensional $L^2$ Jones beta number} $\beta_2(\mu,Q)$ is given by $$\beta_2(\mu,Q):=\inf_\ell \beta_2(\mu,Q,\ell),$$ where the infimum runs over all straight lines in $\RR^n$.

\begin{remark} When $Q=B(x,r)$, we have $$\beta^h_2(\mu,x,r)^2= 4\,\frac{\mu(B(x,r))}{r}\cdot\beta_2(\mu,B(x,r))^2.$$ Thus, the two variants of beta numbers are comparable at scales where $\mu(B(x,r))\sim r$. One advantage of $\beta_2(\mu,B(x,r))$ over $\beta_2^h(\mu,x,r)$ is that the former takes values in $[0,1]$ and when $\mu=\mathscr{L}^n$ is the Lebesgue measure (for example), we have  $$\beta_2(\mathscr{L}^n,B(x,r))\sim 1\quad\text{for all }r>0,$$ while the latter takes values in $[0,\infty)$ and when $\mu=\mathscr{L}^n$ we have  $$\beta_2^h(\mathscr{L}^n,x,r)\sim r^{(n-1)/2}\rightarrow 0\quad\text{as } r\rightarrow 0.$$ We interpret this to mean that for general Radon measures, for which there is no \emph{a priori} control on the coarse density ratios $\mu(B(x,r))/r$, the beta numbers $\beta_2(\mu,B(x,r))$ with non-homogeneous scaling do better at indicating at a glance how well $\mu\res B(x,r)$ can be approximated by a measure supported on a line. The main reason that the numbers $\beta_2^h(\mu,x,r)$ persist in the literature is the historical accident of being initially defined that way by David and Semmes, who made a comprehensive investigation  into boundedness of singular integrals, where it was natural to restrict attention to Ahlfors regular measures (see \cite{DS91,DS93}).\end{remark}

Before we define the anisotropic beta number $\beta^*_2(\mu,Q)$ from \cite{BS3}, we give two lemmas in order to motivate its definition. The following observation, an elegant application of Jensen's inequality, is due to Lerman \cite{Lerman}. Practically speaking, it allows one to control the distance from the $\mu$ center of mass of a window $Q$ to a straight line $\ell$ in terms of the quantity $\beta_2(\mu,Q,\ell)$.

\begin{lemma}[control on the center of mass] \label{l:lerman} Let $\mu$ be a Radon measure on $\RR^n$, let $Q\subset\RR^n$ be a bounded Borel set of positive diameter such that $\mu(Q)>0$, and let $$z_Q:=\int_Q z\,\frac{d\mu(z)}{\mu(Q)}\in \RR^n$$ denote the center of mass of $Q$ with respect to $\mu$. For every straight line $\ell$ in $\RR^n$, $$\dist(z_Q,\ell)\leq \beta_2(\mu,Q,\ell)\diam Q.$$\end{lemma}

\begin{proof} For every affine subspace $\ell$ in $\RR^n$, the function $\dist(\cdot,\ell)^2$ is convex. Thus, $$\dist(z_Q,\ell)^2 = \dist\left(\int_Q z\,\frac{d\mu(z)}{\mu(Q)}\right)^2\leq \int_Q \dist(z,\ell)^2 \frac{d\mu(z)}{\mu(Q)} = \beta_2(\mu,Q,\ell)^2(\diam Q)^2$$ by Jensen's inequality.\end{proof}

Given two windows $R$ and $Q$ with $R\subset Q$, the approximation number $\beta_2(\mu,Q)$ controls the approximation number $\beta_2(\mu,R)$ if one has control on $\diam Q / \diam R$ \textbf{and} $\mu(Q) / \mu(R)$.

\begin{lemma}\label{l:double} Let $\mu$ be a Radon measure on $\RR^n$, let $R,Q\subset\RR^n$ be a bounded Borel set of positive diameter such that $R\subset Q$ and $\mu(R)>0$. For any straight line $\ell$ in $\RR^n$, $$\beta_2(\mu,R,\ell) \leq \left(\frac{\mu(Q)}{\mu(R)}\right)^{1/2}\frac{\diam Q}{\diam R} \,\beta_2(\mu,Q,\ell).$$
\end{lemma}

\begin{proof} This is immediate from monotonicity of the integral. \end{proof}

Now suppose that a window $Q$ contains subregions $R_1,\dots, R_m$ with $\diam R_i \sim \diam Q$. (For example, imagine that $Q$ is a dyadic cube and $R_1,\dots,R_m$ are its dyadic descendants through some fixed number of generations.) Let $z_{R_1},\dots,z_{R_m}$ denote the $\mu$ centers of mass of $R_1,\dots,R_m$, respectively. Roughly speaking, if $\mu$ is a doubling measure, then we have a uniform bound on $\mu(Q)/\mu(R_i)$, and thus, by Lemmas \ref{l:lerman} and \ref{l:double}, $$\dist(z_{R_i},\ell_Q) \leq 2 \beta_2(\mu,Q) \diam Q\quad\text{for all }1\leq i\leq m,$$ where is $\ell_Q$ is any line such that $\beta_2(\mu,Q,\ell_Q)\leq 2 \beta_2(\mu,Q)$. In other words, from just the doubling of $\mu$ and the definition of $\beta_2(\mu,Q)$, we can find a line $\ell_Q$ such that $\beta_2(\mu,Q)$ controls the distance of $z_{R_1},\dots,z_{R_m}$ to $\ell_Q$. For a general Radon measure $\mu$ on $\RR^n$, it is not possible to uniformly bound the ratios $\mu(Q)/\mu(R_i)$, and we cannot hope to control the distance of the points $z_{R_1},\dots,z_{R_m}$ to a common line using the number $\beta_2(\mu,Q)$ alone. Thus, following \cite{BS3}, we modify the definition of $\beta_2(\mu,Q)$ to impose the missing control:

Let $Q\subset\RR^n$ be a dyadic cube. We say that a dyadic cube $R\subset\RR^n$ is \emph{nearby} $Q$ and write $R\in\Delta^*(Q)$ if \begin{itemize}
\item $\diam R=\diam Q$ or $\diam R=\frac{1}{2}\diam Q$; and,
\item the concentric dilate $3R$ is contained in the concentric dilate $1600\sqrt{n} Q$.
\end{itemize} Given a Radon measure $\mu$ on $\RR^n$ and a dyadic cube $Q\subset\RR^n$, we define $\beta_2^*(\mu,Q)$ by $$\beta_2^*(\mu,Q)^2 := \inf_\ell \sup_{R\in\Delta^*(Q)} \beta_2(\mu,R,\ell)^2 \min\left\{\frac{\mu(3R)}{\diam R},1\right\}.$$ Let us temporarily ignore the truncation weight $\min\{\mu(3R)/\diam R,1\}$ and see what we have gained by using nearby cubes $R$ in the definition of $\beta_2^*(\mu,Q)$. Here we are using a dyadic cube $Q\subset\RR^n$ as a convenient short hand to represent a location and scale in $\RR^n$. Each dyadic cube $Q$ can be viewed as the center of a window $1600\sqrt{n}Q$ that contains the subregions $3R$ for all $R\in\Delta^*(Q)$. The beta number $\beta^*(\mu,Q)$ is ``anisotropic" because one normalizes the defining integrals for the $\beta_2(\mu,3R,\ell)$ independently within each subregion of the window $1600\sqrt{n}Q$, whereas the beta number $\beta_2(\mu,1600\sqrt{n}Q)$ for the full window uses a single normalization. From the definition of $\beta_2^*(\mu,Q)$, we can find a line $\ell$ such that $$ \beta_2(\mu,3R,\ell_Q)\min \left\{\left(\frac{\mu(3R)}{\diam R}\right)^{1/2},1\right\} \leq 2\beta_2^*(\mu,Q)$$ for all $R\in\Delta^*(Q)$. Thus, by Lemma \ref{l:lerman}, we can control the distance of $z_{3R}$ to $\ell_Q$ for all subregions $3R$ of $1600\sqrt{n}Q$ corresponding to $R\in\Delta^*(Q)$ (such that $\mu(3R)/\diam R\gtrsim 1$) using $\beta_2^*(\mu,Q)$. That is, we get uniform control of multiple $\mu$ centers of mass to a common line using a single approximation number, without needing to assume that $\mu$ is doubling!

For every Radon measure $\mu$ on $\RR^n$, we now define the \emph{(dyadic) density-normalized Jones function} $J^*_2(\mu,\cdot):\RR^n\rightarrow[0,\infty]$ associated to the numbers $\beta^*(\mu,Q)$  by $$J_2^*(\mu,x) := \sum_{\stackrel{Q\text{ dyadic}}{\side Q\leq 1}} \beta_2^*(\mu,Q)^2 \frac{\diam Q}{\mu(Q)} \chi_Q(x)\quad\text{for all }x\in\RR^n.$$ Suppose that $Q_0$ is a dyadic cube of side length at most 1 and let $N<\infty$. Integrating $J_2^*(\mu,x)$ on the set $\{x\in Q_0: J_2^{*}(\mu,x)\leq N\}$ with respect to $\mu$, one obtains $$\sum_{Q\in\mathcal{T}} \beta_2^*(\mu,Q)^2\diam Q\leq N\mu(Q_0)<\infty,$$ where $\mathcal{T}$ is the tree of dyadic cubes $Q\subset Q_0$ such that $$Q\cap \{x\in Q_0: J_2^{*}(\mu,x)\leq N\}\neq\emptyset.$$ For each $c>0$, let $\mathcal{T}_c$ be a maximal subtree where $\mu(3Q)/\diam Q\geq c$ for all $Q\in\mathcal{T}_c$. The argument that we sketched in the previous paragraph lets us control the distance of the $\mu$ centers of mass of triples of nearby cubes in $\mathcal{T}_c$ in terms of the numbers $\beta_2^*(\mu,Q)$. A technical extension of Jones' original traveling salesman construction proved in \cite{BS3} then lets us build a rectifiable curve $\Gamma$ containing the leaves of the tree $\mathcal{T}_c$. (The constant $1600\sqrt{n}$ and number of generations of dyadic cubes (2) appearing in the definition of nearby cubes is chosen so that there are sufficiently many data points to enter into the extended traveling salesman construction, see \cite[Proposition 3.6]{BS3}.) Running this procedure over a suitable countable choice of parameters, we arrive at Badger and Schul's solution to the identification problem for 1-rectifiable Radon measures:

\begin{theorem}[see {\cite[Theorem A]{BS3}}]\label{t:big} Let $\mu$ be a Radon measure on $\RR^n$. Then $$\mu^1_\rect = \mu\res\{x\in\RR^n:\lD1(\mu,x)>0\text{ and }J_2^*(\mu,x)<\infty\},$$ $$\mu^1_\pu = \mu\res\{x\in\RR^n : \lD1(\mu,x)=0\text{ or }J_2^*(\mu,x)=\infty\}.$$
\end{theorem}

\begin{remark}While Theorem \ref{t:big} cannot be directly applied to $\sigma$-finite Borel measures that are not locally finite, we note that it can be used indirectly, as follows. Suppose that $\nu$ is an infinite, $\sigma$-finite Borel measure on $\RR^n$. Then there exists a countable measurable partition of $\RR^n$ into disjoint Borel sets $\{X_i\}_1^\infty$ such that $0<\nu(X_i)<\infty$ for all $i\geq 1$. Suppose we can find such a partition. Then the measure $$\rho:=\sum_{i=1}^\infty \frac{1}{2^i\nu(X_i)}\nu\res X_i$$ is a Borel probability measure such that $\nu\ll\rho\ll \nu$. Theorem \ref{t:big} identifies a Borel set $A_\rho$ such that $\rho^1_\rect = \rho\res A_\rho$ and $\rho^1_\pu=\rho\res(\RR^n\setminus A_\rho)$. Since $\nu\ll\rho$, it follows that $$\nu^1_\rect=\nu\res A_\rho\quad\text{and}\quad\nu^1_\pu = \nu\res (\RR^n\setminus A_\rho),$$ as well (which we leave as a simple exercise for the reader). In this sense, because we can solve the identification problem for any finite Borel measure $\rho$ on $\RR^n$, we can also solve the problem for any $\sigma$-finite Borel measure $\nu$ on $\RR^n$.\end{remark}

For every Radon measure $\mu$ on $\RR^n$ and dyadic cube $Q\subset\RR^n$, one can define a variant $\beta^{**}_2(\mu,Q)$ of $\beta^*(\mu,Q)$ without the truncation weight by $$\beta_2^{**}(\mu,Q) := \inf_\ell \sup_{R\in\Delta^*(Q)} \beta_2(\mu,3R,\ell).$$ Also define the associated density-normalized Jones function $J^{**}_2(\mu,\cdot):\RR^n\rightarrow[0,\infty]$ by $$J_2^{**}(\mu,x) := \sum_{\stackrel{Q\text{ dyadic}}{\side Q\leq 1}} \beta_2^{**}(\mu,Q)^2 \frac{\diam Q}{\mu(Q)} \chi_Q(x)\quad\text{for all }x\in\RR^n.$$ The comparisons $\beta_2^*(\mu,Q)\leq \beta_2^{**}(\mu,Q)$ and $J_2^*(\mu,x)\leq J_2^{**}(\mu,x)$ hold for all $\mu$, $Q$, and $x$. In \cite{BS3}, the authors proved the following theorems, which also follow from the argument outlined above. (In fact, R. Schul and I discovered Theorem \ref{t:star2} before Theorem \ref{t:big}.)

\begin{theorem}[see {\cite[Theorem D]{BS3}}; the Traveling Salesman Theorem for Measures] \label{t:tstm} Let $\mu$ be a finite Borel measure on $\RR^n$ with bounded support. Then $\mu$ is carried by a rectifiable curve (i.e.~there is a rectifiable curve $\Gamma$ such that $\mu(\RR^n\setminus\Gamma)=0$) if and only if $$S_2^{**}(\mu):=\sum_{Q\text{ dyadic}} \beta_2^{**}(\mu,Q)^2\diam Q<\infty.$$ Moreover, if $S_2^{**}(\mu)<\infty$, then there is a rectifiable curve $\Gamma$ in $\RR^n$ carrying $\mu$ such that $\Haus^1(\Gamma)\lesssim_n \diam(\spt\mu) + S_2^{**}(\mu)$\end{theorem}

\begin{theorem}[see {\cite[\S5]{BS3}}] \label{t:star2} Let $\mu$ be a Radon measure on $\RR^n$. Then we have $\mu\res \left\{x: J_2^{**}(\mu,x)<\infty\right\}$ is 1-rectifiable. (Hence $\mu\res \left\{x: J_2^{**}(\mu,x)<\infty\right\}\leq\mu^1_\rect$.)
\end{theorem}

One difference between Theorem \ref{t:big} and Theorem \ref{t:star2} is that the latter does not require a bound on the lower Hausdorff density of $\mu$. It is currently an open problem to decide whether or not for every Radon measure $\mu$ on $\RR^n$, $$J_2^{**}(\mu,x)<\infty\quad\text{at $\mu^1_\rect$-a.e.~$x\in\RR^n$}.$$ In other words, it is currently unknown whether or not the sufficient condition for a measure to be 1-rectifiable in Theorem \ref{t:star2} is also a necessary condition. To overcome this gap and obtain a characterization, R. Schul and I (motivated by our earlier work) inserted the truncation weight $$\min\{\mu(3R)/\diam R,1\}$$ into the definition of $\beta^{**}(\mu,Q)$ to obtain $\beta^*(\mu,Q)$. This was plausible from the point of view of retaining the sufficient condition for rectifiability established in Theorem \ref{t:star2} by Lemma \ref{l:larp}. With the addition of the truncation weight, one can then follow the proof of the main theorem in \cite{BS} (using the ``necessary" half of Theorem \ref{t:tst}) and show that $$\int_\Gamma J_2^*(\mu,x)\,d\mu(x)<\infty$$ for any rectifiable curve $\Gamma$ in $\RR^n$, from which it follows that $J_2^*(\mu,x)<\infty$ at $\mu^1_\rect$-a.e.~$x$. See \cite[\S4]{BS3} for details. The idea of discounting the linear approximation numbers at scales of small density also appears in the recent work  of Naber and Valtorta \cite{NV} on a measure-theoretic version of the Reifenberg algorithm.

\begin{remark} The resolution of the identification problem for 1-rectifiable Radon measures provides a template for attacking similar problems. To solve the identification problem for Radon measures carried by a family of sets $\mathcal{N}$ in a space $\XX$, there are three basic steps.
First, find a suitable characterization of \emph{subsets} of the sets in $\mathcal{N}$. Second, transform the result for sets into a characterization of \emph{doubling measures} carried by $\mathcal{N}$. Third, introduce \emph{anisotropic normalizations} to promote the characterization for doubling measures to a characterization for Radon measures. For example, I expect it should be possible to follow this plan to solve the identification problem for Radon measures in $\RR^n$ that are carried by $m$-dimensional Lipschitz graphs.\end{remark}

\section{Recent progress on higher dimensional rectifiability}

Throughout this section, we let $m$ and $n$ denote integers with $1\leq m\leq n-1$. Depending on the context, one encounters three possible definitions of $m$-rectifiable measures in $\RR^n$, which coincide for $m$-dimensional Hausdorff measures on $m$-sets, but differ for locally finite measures in general.  Let $\mu$ be a Radon measure on $\RR^n$. In decreasing order of generality, we say that (using the terminology in Definition \ref{def:carry}) \begin{enumerate}
\item $\mu$ is \emph{(Lipschitz) image $m$-rectifiable} if $\mu$ is carried by images of Lipschitz maps $f:[0,1]^m\rightarrow\RR^n$;
\item $\mu$ is \emph{(Lipschitz) graph $m$-rectifiable} if $\mu$ is carried by isometric copies of graphs of Lipschitz maps $g:\RR^{m}\rightarrow\RR^{n-m}$;
\item $\mu$ is \emph{$C^1$ $m$-rectifiable} if $\mu$ is carried by $m$-dimensional embedded $C^1$ submanifolds of $\RR^n$.
\end{enumerate} That is, $(3)\Rightarrow(2)\Rightarrow(1)$ for every Radon measure $\mu$, and moreover, both implications are trivial.
For restrictions of $m$-dimensional Hausdorff measures to $m$-sets and absolutely continuous measures, the three notions of $m$-dimensional rectifiability are equivalent:

\begin{theorem}\label{t:equiv} Suppose $E\subset\RR^n$ is $\mathcal{H}^m$ measurable and $0<\Haus^m(E)<\infty$. If $\mu=\Haus^m\res E$ is image or graph $m$-rectifiable, then $\mu$ is $C^1$ $m$-rectifiable. More generally, if $\mu$ is a Radon measure on $\RR^n$ and $\mu\ll\Haus^m$, then $\mu$ is $m$-rectifiable with respect to one of the definitions (1), (2), or (3) if and only if $\mu$ is $m$-rectifiable with respect to each of (1), (2), and (3).\end{theorem}

\begin{proof} For the first part, see Federer \cite[3.2.29]{Federer}. The second part follows from the first with the aid of the Radon-Nikodym theorem.\end{proof}

In \cite{MM1988}, Mart\'in and Mattila constructed, for all $0<s<1$, sets $E_s,F_s\subset\RR^2$ with $$0<\Haus^s(E_s)<\infty\quad\text{and}\quad0<\Haus^s(F_s)<\infty$$  such that $\Haus^s\res E_s$ is image $1$-rectifiable, but not graph 1-rectifiable, and $\Haus^s\res F_s$ is graph 1-rectifiable, but not $C^1$ 1-rectifiable. Thus, when working with general Radon measures, which are not necessarily absolutely continuous with respect to $\Haus^m$, it is important to distinguish between image, graph, and $C^1$ rectifiability.

\begin{remark}\label{rk:why} It would be nice if there were a universal convention about which version of rectifiability (image, graph, or $C^1$) is the default definition. My own preference is that image rectifiability should be default, and henceforth, will say simply that a Radon measure $\mu$ on $\RR^n$ is \emph{$m$-rectifiable} if $\mu$ is image $m$-rectifiable. There are three basic reasons. First, image rectifiability is the definition that is consistent with the convention that a $1$-rectifiable measure is a  measure carried by rectifiable curves, which is used implicitly and explicitly in the work of Besicovitch \cite{Bes28,Bes38} and Morse and Randolph \cite{MR}. Second, in the monograph \cite{Federer} that gave the field of geometric measure theory its name, Federer makes the definition that \emph{$E\subset\RR^n$ is countably $(\mu,m)$ rectifiable} if $\mu\res E$ is image $m$-rectifiable in the sense above. (That is, even though Federer defines rectifiability as a property of a set rather than as a property of a measure, he uses Lipschitz images in the definition.) Third, of the three definitions, image rectifiability is the most general. For a setting where graph rectifiability is the most natural definition, see \cite{7author}. \end{remark}

For any Radon measure $\mu$ on $\RR^n$ and integer $1\leq m\leq n-1$, let $\mu=\mu^m_\rect+\mu^m_\pu$ denote the decomposition from Proposition \ref{p:decomp}, where $\mu^m_\rect$ is carried by images of Lipschitz maps $f:[0,1]^m\rightarrow\RR$ and $\mu^m_\pu$ is singular to images Lipschitz maps $f:[0,1]^m\rightarrow\RR^n$ in the sense of Definition \ref{def:carry}. The following fundamental problem in geometric measure theory about the structure of measures in Euclidean space is wide open.

\begin{problem}[identification problem for $m$-rectifiable measures] \label{p:hd} Let $2\leq m\leq n-1$. Find geometric or measure-theoretic properties that identify $\mu^m_\rect$ and $\mu^m_\pu$ for every Radon measure $\mu$ on $\RR^n$. (Do not assume that $\mu\ll\Haus^m$.) \end{problem}

Given a Radon measure $\mu$ on $\RR^n$ and $s>0$, the lower and upper Hausdorff $s$-densities of $\mu$ are defined by $$\lD{s}(\mu,x):=\liminf_{r\downarrow 0} \frac{\mu(B(x,r))}{\omega_s r^s}\quad\text{and}\quad \uD{s}:=\limsup_{r\downarrow 0} \frac{\mu(B(x,r))}{\omega_s r^s}$$ for all $x\in\RR^n$, where $\omega_s$ is a constant depending on the normalization used in the definition of the $s$-dimensional Hausdorff measure. In particular, $\omega_s=2^s$ when  $\mathcal{H}^s$ is defined by $$\Haus^s(E)=\lim_{\delta\downarrow 0} \inf \left \{\sum_{i=1}^\infty (\diam E_i)^s: E\subset\bigcup_{i=1}^\infty E_i\text{ and }\diam E_i\leq \delta\right\}.$$
As with the case $m=1$, rectifiability of a measure restricts the appropriate lower density.

\begin{lemma}[see {\cite[Lemma 2.7]{BS}}] \label{l:marp} Let $\mu$ be a Radon measure on $\RR^n$. Then $$\mu^m_\rect \leq \mu\res\{x: \lD{m}(\mu,x)>0\}\quad\text{ and }\quad \mu\res\{x:\lD{m}(\mu,x)=0\}\leq \mu^m_\pu.$$ That is, $\mu\res\{x:\lD{m}(\mu,x)=0\}$ is purely $m$-unrectifiable, and if $\mu$ is $m$-rectifiable, then $\lD{m}(\mu,x)>0$ at $\mu$-a.e. $x\in\RR^n$. \end{lemma}

For absolutely continuous measures,  which are the Radon measures such that $\mu\ll\Haus^m$, or equivalently (see the exercises in \cite[Chapter 6]{Mattila}), the Radon measures such that $$\uD{m}(\mu,x)=\limsup_{r\downarrow 0} \frac{\mu(B(x,r))}{\omega_m r^m}<\infty\quad\text{at $\mu$-a.e. $x\in\RR^n$},$$ several solutions to Problem \ref{p:hd} are available. At a high level, one of the reasons that Problem \ref{p:hd} is easier for absolutely continuous measures is that for such measures the notions of Lipschitz image, Lipschitz graph, and $C^1$ rectifiability coincide. The following is the combined effort of several mathematicians (see below for a detailed citation).

\begin{theorem}[assorted characterizations of absolutely continuous rectifiable measures] \label{t:assort}Let $1\leq m\leq n-1$ be integers and let $\mu$ be a Radon measure on $\RR^n$. Assume $\mu\ll\Haus^m$ ($\Leftrightarrow$ $\uD{m}(\mu,x)<\infty$ at $\mu$-a.e.~$x\in\RR^n$.) Then the following are equivalent: \begin{enumerate}
\item $\mu$ is $m$-rectifiable (carried by Lipschitz images of $[0,1]^m$);
\item there is a unique $\mu$ approximate tangent $m$-plane at $\mu$-a.e. $x\in\RR^n$;
\item $\mu$ is weakly $m$-linearly approximable at $\mu$-a.e.~$x\in\RR^n$, i.e.~ $\lD{m}(\mu,x)>0$ and $$\mathrm{Tan}(\mu,x)\subseteq \{c\mathcal{H}^m\res L:c>0, L\in G(n,m)\}\quad\text{at $\mu$-a.e.~$x\in\RR^n$};$$
\item the $m$-dimensional Hausdorff density of $\mu$ exists and is positive $\mu$-a.e.: $$\liminf_{r\downarrow 0} \frac{\mu(B(x,r))}{\omega_m r^m}=\limsup_{r\downarrow 0}\frac{\mu(B(x,r))}{\omega_m r^m}>0\quad\text{at $\mu$-a.e. $x\in\RR^n$};$$
\item $\lD{m}(\mu,x)>0$ at $\mu$-a.e. $x\in\RR^n$ and the coarse density ratios are asymptotically optimally doubling in the sense that $$\lim_{r\downarrow 0}\frac{\mu(B(x,2r))}{\omega_m(2r)^m}-\frac{\mu(B(x,r))}{\omega_m r^m}=0\quad\text{ at $\mu$-a.e. $x\in\RR^n$};$$
\item $\uD{m}(\mu,x)>0$ at $\mu$-a.e. $x\in\RR^n$ and the Jones function associated to the $m$-dimensional homogeneous $L^2$ beta numbers $\beta_2^{m,h}(\mu,x,r)$ is finite $\mu$-a.e.: $$\int_0^1 \beta_2^{m,h}(\mu,x,r)^2\frac{dr}{r}<\infty\quad\text{ at $\mu$-a.e. $x\in\RR^n$.}$$
\end{enumerate}
\end{theorem}

\begin{proof} First, let us explain the various bits of notation and terminology used to state the theorem. We say that an $m$-dimensional affine subspace $L$ of $\RR^n$ containing $x\in\RR^n$ is a \emph{$\mu$ approximate tangent $m$-plane at $x$} if $\uD{m}(\mu,x)>0$ and $$\limsup_{r\downarrow 0} \frac{\mu(B(x,r)\setminus X)}{\omega_m r^m}=0$$ for every cone of the form $X=\{y: \dist(y,L)\leq  c \dist(y,L^\perp)\}$, where $c>0$ and $L^\perp$ denotes the unique $(n-m)$-plane  containing $x$ that is orthogonal to $L$. Following \cite{Preiss}, we say that a non-zero measure $\nu$ is a \emph{tangent measure} of $\mu$ at $x$ if there exist sequences $r_i\downarrow 0$ and $0<c_i<\infty$ such that the rescaled measures $\mu_i$, defined by $$\mu_i(E)=c_i\mu(x+r_iE),$$ converge weakly to $\nu$ in the sense of Radon measures. We denote the set of all tangent measures of $\mu$ at $x$ by $\mathrm{Tan}(\mu,x)$. As usual, the set $G(n,m)$ denotes the Grassmannian of $m$-dimensional linear subspaces of $\RR^n$. Following \cite{DS91}, the homogeneous $m$-dimensional $L^2$ Jones beta number $\beta_2^{m,h}(\mu,x,r)$ is defined by \begin{equation}\label{eq:beta-m}\beta_2^{m,h}(\mu,x,r)^2:=\inf_L \int_{B(x,r)} \left(\frac{\dist(x,L)}{r}\right)^2\frac{d\mu(x)}{r^m},\end{equation} where the infimum runs over all $m$-dimensional affine subspaces of $\RR^n$.

The implications $(1)\Rightarrow(2)\Rightarrow(4)$ essentially follow from Rademacher's theorem on almost everywhere differentiability of Lipschitz maps and may be considered elementary. The implication $(2)\Rightarrow(1)$ was proved by Federer \cite{Fed47}. The implication $(2)\Rightarrow(3)$ is trivial. The implication $(3)\Rightarrow(1)$ was proved by Mattila \cite{Mattila75} (extending the special case $m=2$ and $n=3$ by Marstrand \cite{Marstrand61}). The implication $(4)\Rightarrow(3)$ was proved by Preiss \cite{Preiss}. Thus, the equivalence of (1), (2), (3), and (4) was finally completed after a span of 40 years from the 1940s to the 1980s. The essential difficulty in the final implication $(4)\Rightarrow(3)$ is the (surprising!) existence of non-flat $m$-uniform measures in $\RR^n$, i.e.~Radon measures in $\RR^n$ such that $$\mu(B(x,r))=cr^m\quad\text{for all $x\in\spt\mu$ and $r>0$},$$ for which the support of $\mu$ is \emph{not} an $m$-plane. For further background and the most recent developments on the uniform measure classification problem, including new examples of $3$-uniform measures in $\RR^n$, for each $n\geq 5$, see Nimer \cite{nimer3}. For a friendly presentation of the proof of $(1)\Leftrightarrow(2)\Leftrightarrow(3)\Leftrightarrow(4)$,  see the monograph \cite{DeLellis} by De Lellis.

The equivalence of (5) and (6) with (1), (2), (3), and (4) are recent developments. The implication $(4)\Rightarrow(5)$ is trivial. The implication $(5)\Rightarrow(4)$ was proved by Tolsa and Toro \cite{TT-rect}. The implication $(1)\Rightarrow(6)$ was proved by Tolsa \cite{Tolsa-n}. With the additional assumption $\lD{m}(\mu,x)>0$ $\mu$-a.e., the implication $(6)\Rightarrow(1)$ was proved by Pajot \cite{Pajot97} (for Radon measures of the form $\mu=\Haus^m\res E$, $E$ compact) and Badger and Schul \cite{BS2}. With the \emph{a priori} weaker assumption $\uD{m}(\mu,x)>0$ $\mu$-a.e., the implication $(6)\Rightarrow(1)$ was proved by Azzam and Tolsa \cite{AT}. To obtain the last result, Azzam and Tolsa carry out an intricate stopping time argument using David and Mattila's ``dyadic" cubes \cite{DMcubes} and David and Toro's extension of Reifenberg's algorithm to sets with holes \cite{DT}.
\end{proof}

For related work on rectifiability of  absolutely continuous measures and the theory of mass transport, see Tolsa \cite{Tolsa12}, Azzam, David, and Toro \cite{ADT1,ADT2}, and Azzam, Tolsa, and Toro \cite{ATT-alphas}. For the connection between rectifability of sets and Menger-type curvatures, see  L\'eger \cite{Leger}, Lerman and Whitehouse \cite{LW-I,LW-II}, Meurer \cite{Meurer}, and Goering \cite{max-menger}. For related results about discrete approximation and rectifiability of varifolds, see Buet \cite{Buet15}.
\bigskip

In recent work \cite{ENV}, Edelen, Naber, and Valtorta provide new sufficient conditions for qualitative and quantitative rectifiability of Radon measures $\mu$ on $\RR^n$ that do not require absolute continuity, $\mu\ll\Haus^m$ ($\Leftrightarrow \uD{m}(\mu,x)<\infty$ at $\mu$-a.e.~$x\in\RR^n$.) For simplicity, we state a special case of their main result.

\begin{theorem}[see Edelen-Naber-Valtorta {\cite{ENV}}] \label{t:env} Let $\mu$ be a Radon measure on $\RR^n$. Assume that $\lD{m}(\mu,x)<\infty$ and $\uD{m}(\mu,x)>0$ at $\mu$-a.e. $x\in\RR^n$. Then $$\int_0^1 \beta_2^{m,h}(\mu,x,r)^2\frac{dr}{r}<\infty\quad\text{ at $\mu$-a.e. $x\in\RR^n$}$$ implies $\mu$ is $m$-rectifiable (and hence $\lD{m}(\mu,x)>0$ $\mu$-a.e.)\end{theorem}

\begin{remark} Edelen, Naber, and Valtorta's proof of Theorem \ref{t:env} is based on an updated, quantitative version of the Reifenberg algorithm (cf.~\cite{DT} and \cite{NV}), the original version of which allows one to parameterize sets which are sufficiently ``locally flat" at all locations and scales. A second proof of Theorem \ref{t:env} has been provided by Tolsa \cite{Tolsa-betas} using a method closer to that of \cite{AT}. It is evident that the sufficient condition for a measure to be $m$-rectifiable in Theorem \ref{t:env} is not necessary in view of the example of Martikainen and Orponen \cite{MO} from the case $m=1$. \end{remark}

At the root, the main challenge in solving Problem \ref{p:hd} is the lack of a good Lipschitz parameterization theorem for sets of dimension $m\geq 2$. If $f:[0,1]^m\rightarrow\RR^n$ is Lipschitz, then $\Gamma=f([0,1]^m)$ is compact, connected, locally connected, and $\mathcal{H}^m(\Gamma)<\infty$. When $m=1$, the converse of this fact is also true. In fact, it has been known for over 90 years that if $\Gamma\subset\RR^n$ is a nonempty continuum with $\Haus^1(\Gamma)<\infty$, then $\Gamma$ is a Lipschitz image of $[0,1]$ (see \cite{AO} where the theorem is attributed to a paper of Wa\.{z}ewski from 1927, and for a simple proof, see the appendix of \cite{Schul-Hilbert}.) The situation in higher dimensions is quite different. For example, let $C\subset[0,1]^2$ be any self-similar Cantor set of Hausdorff dimension 1 (so that $\Haus^1\res C$ is purely 1-unrectifiable by Hutchinson \cite[\S5.4]{Hutch}.) Then $\Haus^2\res (C\times[0,1])$ is purely 2-unrectifiable by Theorem \ref{t:assort}(2). By adjoining a square to the base of $C\times[0,1]$, we obtain a set $K$, $$K:=(C\times[0,1])\cup([0,1]^2\times\{0\})\subset\RR^3,$$ which is compact, connected, path-connected, and Ahlfors 2-regular (hence $\Haus^2(\Gamma)<\infty$), but $K$ is not locally connected, and therefore, $K$ is not a Lipschitz image of $[0,1]^2$ nor a continuous image of $[0,1]$. From this one sees that being a continuum with finite or Ahlfors regular $\Haus^2$ measure does not ensure a set is locally connected. In fact, it turns out that even with an added assumption of local connectedness, one still cannot guarantee the existence of a Lipschitz parameterization for Ahlfors regular contiuum when $m\geq 2$. The following result will appear in forthcoming work by the author, Naples, and Vellis.

\begin{proposition}[see {\cite[\S9.2]{BNV}}; ``Cantor ladders"] For all $m\geq 2$, there exist compact, connected, locally connected, $m$-Ahlfors regular sets $G\subset\RR^n$ such that $G$ is not contained in the image of any Lipschitz map $f:[0,1]^m\rightarrow\RR^{m+1}$.\end{proposition}

\begin{proof}[Proof Sketch ($m=2$)] Start with the set $K$ above. One can transform $K$ into a locally connected set $G$ by adjoining sufficiently many squares $S_{j,k}\times\{t_k\}\subset\RR^2\times\RR$, $1\leq j\leq J_k$ on a countable dense set of heights $t_k\in [0,1]$, with diameters of $S_{j,k}$ vanishing as $k\rightarrow\infty$. By carefully selecting parameters, one can ensure that $G$ is Ahlfors $2$-regular. However, since $G$ contains $C\times[0,1]$ and $\Haus^2\res(C\times[0,1])$ is purely 2-unrectifiable, it is not possible that $G$ is contained in the image of a Lipschitz map $f:[0,1]^2\rightarrow\RR^3$.\end{proof}

In order to attack Problem \ref{p:hd}, it would be useful to first have new sufficient criteria for identifying Lipschitz images. In the author's view, any solution of the following problem would be interesting (even if very far from a necessary condition).

\begin{problem} \label{prob:squares} For each $m\geq 2$, find sufficient geometric, metric, and/or topological conditions that ensure a set $\Gamma\subset\RR^n$ is (contained in) a Lipschitz image of $[0,1]^m$.\end{problem}

Lipschitz images of $[0,1]^m$ represent only one possible choice of model sets to build a theory of higher dimensional rectifiability, and it may be worthwhile to explore other families of sets for which it is possible to solve Problem \ref{prob:ident}. One promising alternative is a class of surfaces that support a traveling salesman type theorem, recently identified by Azzam and Schul \cite{AS}, which are lower regular with respect to the \emph{$m$-dimensional Hausdorff content} $\Haus^m_\infty(E):=\inf\{\sum (\diam E_i)^m:E\subset\bigcup_i E_i\}$. For a complete description, we refer the reader to \cite{AS}. Also see Villa \cite{Villa}, which characterizes the existence of approximate tangent $m$-planes of content lower regular sets.



\section{Fractional rectifiability and other frontiers}

\subsection{Fractional rectifiability}

Let $1\leq m\leq n-1$ be integers and let $s$ be a real number with $s\geq m$. Recall that a map $f:[0,1]^m\rightarrow\RR^n$ is $(m/s)$-H\"older continuous if there is a constant $H<\infty$ such that $$|f(x)-f(y)|\leq H|x-y|^{m/s}\quad\text{for all }x,y\in[0,1]^m.$$ Alternatively, one may view the map $f$ as a Lipschitz map from the snowflaked metric space $([0,1]^m,|\cdot|^{m/s})$ to $\RR^n$. Below we call the image $\Gamma=f([0,1]^m)$ of such a map a \emph{$(m/s)$-H\"older $m$-cube}, or when $m=1$, we may call $\Gamma$ a \emph{$(1/s)$-H\"older curve}.
From the definition of Hausdorff measures, it easily follows that $$\Haus^s(\Gamma)\lesssim_{m,s,H} \Haus^m ([0,1]^m)<\infty.$$ Thus, every $(m/s)$-H\"older $m$-cube $\Gamma\subset\RR^n$ is a compact, connected, locally connected set with $\Haus^s(\Gamma)<\infty$. Mart\'in and Mattila \cite{MM1993} suggested using images of $(m/s)$-H\"older maps from $\RR^m$ as a basis for studying ``fractional rectifiability" of $\Haus^s$ measures restricted to \emph{$s$-sets} $E\subset\RR^n$, i.e.~ $\Haus^s$ measurable sets such that $0<\Haus^s(E)<\infty$. In \cite{BV}, V. Vellis and I revisited Mart\'in and Mattila's idea in the context of arbitrary Radon measures:

By Proposition \ref{p:decomp}, any Radon measure $\mu$ on $\RR^n$ can be decomposed as $$\mu=\mu_{m\rightarrow s}+\mu_{m\rightarrow s}^\perp,$$ where $\mu_{m\rightarrow s}$ is carried by $(m/s)$-H\"older $m$-cubes and $\mu_{m\rightarrow s}^\perp$ is singular to $(m/s)$-H\"older $m$-cubes. When $s=m$, the measures $$\mu_{m\rightarrow m}=\mu^{m}_\rect\quad\text{and}\quad\mu_{m\rightarrow m}^\perp=\mu^m_{\pu}$$ are the $m$-rectifiable and purely $m$-unrectifiable parts of $\mu$, respectively.
The existence of $(m/n)$-H\"older surjections $f:[0,1]^m\rightarrow[0,1]^n$, commonly called \emph{space-filling maps} (see \cite[Chapter VII, \S3]{ss-reals} for the case $m=1$ and \cite[\S9.1]{semmes-buffalo}, \cite{stong} for $m\geq 2$),  guarantees that $\mu_{m\rightarrow s}=\mu$ and $\mu_{m\rightarrow s}^\perp=0$ for all $s\geq n$. Also, Theorem \ref{t:big} above identifies $\mu_{1\rightarrow 1}$ and $\mu_{1\rightarrow 1}^\perp$ for all Radon measures. All other cases of the following problem are open.

\begin{problem}[identification problem for fractional rectifiability] \label{p:fd} Let $1\leq m\leq n-1$ be integers and let $s\in[m,n)$. Find geometric or measure-theoretic properties that identify $\mu_{m\rightarrow s}$ and $\mu_{m\rightarrow s}^\perp$ for every Radon measure $\mu$ on $\RR^n$. (Do not assume that $\mu\ll\Haus^s$.) \end{problem}

Hausdorff measures on self-similar sets provide essential examples of rectifiable and purely unrectifiable behavior in fractional dimensions. For further results in this direction, see Mart\'in and Mattila \cite{MM2000} and Rao and Zhang \cite{Rao-Zhang}.

\begin{theorem}[see \cite{MM1993}] \label{t:mm} Let $1\leq m\leq n-1$ and let $S=\bigcup_{i=1}^k f_i(S)\subset\RR^n$ be a self-similar set of Hausdorff dimension $s\in[m,n)$ in the sense of Hutchinson \cite{Hutch}. Assume that $$f_i(S)\cap f_j(S)=\emptyset\quad\text{for all }i\neq j,$$ Then $\Haus^s\res S$ is singular to $(m/s)$-H\"older $m$-cubes.
\end{theorem}

\begin{theorem}[see Remes \cite{Remes}] Let $S=\bigcup_1^kf_i(S)\subset\RR^n$ be a self-similar set of Hausdorff dimension $s\in[1,n)$ that satisfies the open set condition. If $S$ is compact and connected, then $S$ is a $(1/s)$-H\"older curve.\end{theorem}

For Radon measures, extreme behavior of the lower and upper Hausdorff densities force a measure to carried by or singular to $(1/s)$-H\"older curves and  $(m/s)$-H\"older $m$-cubes.

\begin{theorem}[see Badger-Vellis \cite{BV}; behavior under extreme Hausdorff densities] \label{t:bv} Let $1\leq m\leq n-1$ be integers, let $s\in[m,n)$, and let $t\in[0,s)$. If $\mu$ is a Radon measure on $\RR^n$, then \begin{enumerate}
\item $\mu\res\left\{x:\lD{s}(\mu,x)=0\right\}$ is singular to  $(m/s)$-H\"older $m$-cubes;
\item $\nu$ is carried by $(1/s)$-H\"older curves, where $$\nu\equiv\mu\res\left\{x:\int_0^1 \frac{r^s}{\mu(B(x,r))}\frac{dr}{r}<\infty\text{ and }\limsup_{r\downarrow 0}\frac{\mu(B(x,2r))}{\mu(B(x,r))}<\infty\right\};$$
\item $\rho\equiv \mu\res\{x:0<\lD{t}(\mu,x)\leq\uD{t}(\mu,x)<\infty\}$ is carried by $(m/s)$-H\"older $m$-cubes;
\item moreover, if $0\leq t<1$, then $\rho$ is carried by bi-Lipschitz embeddings $f:[0,1]\rightarrow\RR^n$.
\end{enumerate}
\end{theorem}

\begin{example}[$2^n$-corner Cantor Sets] Let $E_t\subset\RR^n$ be a $2^n$-corner Cantor set of Hausdorff dimension $0<t<n$, generated starting from a cube of side length 1 by replacing it with $2^n$ subcubes of side length $2^{-n/t}$, placed at the corners of the original cube (see Figure \ref{fig:8} for a depiction of the case $n=3$). Then $\Haus^t\res E_t$ is Ahlfors $t$-regular, i.e.~ $\Haus^t(E_t\cap B(x,r))\sim r^t$ for all $x\in E_t$ and $0<r\leq 1$. Thus, for any integer $1\leq m\leq n-1$ such that $t\geq m$, \begin{itemize}
\item $\Haus^t\res E_t$ is singular to $(m/t)$-H\"older $m$-cubes by Theorem \ref{t:mm}; and,
\item $\Haus^t\res E_t$ is carried by $(m/s)$-H\"older $m$-cubes for all $s>t$ by Theorem \ref{t:bv}(3).\end{itemize} Moreover, if $0<t<1$, then $\Haus^t\res E_t$ is carried by bi-Lipschitz curves by Theorem \ref{t:bv}(4).
\end{example}

\begin{figure}\begin{center}\includegraphics[width=.5\textwidth]{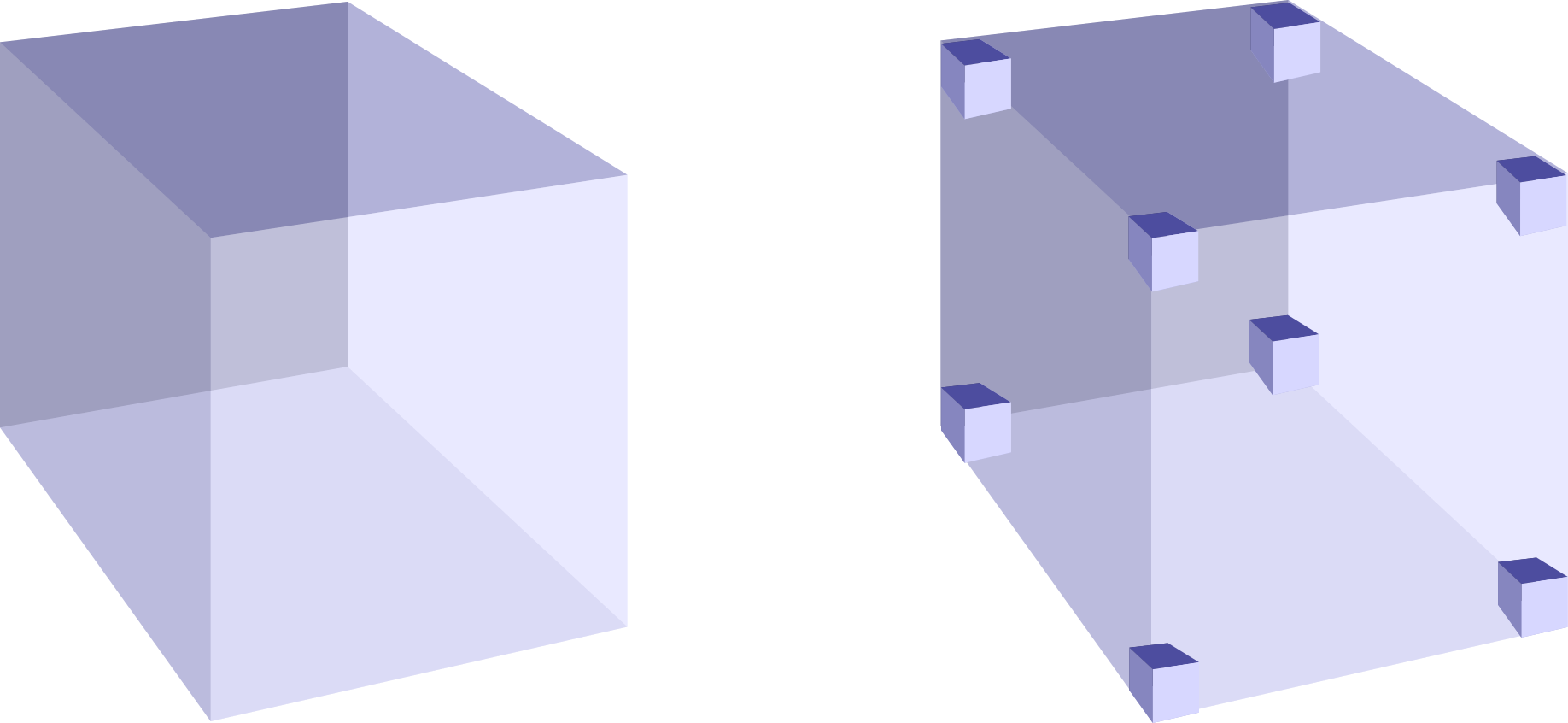}\end{center}
\caption{Generators for an Eight Corner Cantor set} \label{fig:8}
\end{figure}

The following problem is open, but from the point of view of topological dimension may be more tractable than the corresponding problem for Lipschitz squares (see Problem \ref{prob:squares}).

\begin{problem} For all $1<s<2$, find sufficient geometric, metric, and/or topological conditions that ensure a set $\Gamma\subset\RR^n$ is (contained in) a $(1/s)$-H\"older curve.\end{problem}

\subsection{Higher-order rectifiability}

Let $1\leq m\leq n-1$ and $k\geq 1$ be integers, and let $\alpha\in[0,1]$. We say that a Radon measure $\mu$ on $\RR^n$ is \emph{$C^{k,\alpha}$ $m$-rectifiable} if $\mu$ is carried by $m$-dimensional $C^{k,\alpha}$ embedded submanifolds of $\RR^n$. In the case $\alpha=0$, we also say that $\mu$ is $C^k$ $m$-rectifiable. The study of higher-order rectifiability of measures was initiated by Anzellotti and Serapioni \cite{AS94}. In general, different orders of rectifiability give rise to different classes of measures.

\begin{theorem} Let $1\leq m\leq n-1$ and $k,l\geq 1$ be integers, and let $\alpha,\beta\in[0,1]$. \begin{enumerate}
\item If $k+\alpha < l+\beta$, then there exists an $m$-set $E\subset\RR^n$ such that $\Haus^m\res E$ is $C^{k,\alpha}$ $m$-rectifiable and purely $C^{l,\beta}$ $m$-unrectifiable (i.e.~singular to $C^{l,\beta}$ submanifolds).
\item If $\mu$ is a Radon measure on $\RR^n$ and $\mu\ll\Haus^m$, then $\mu$ is $C^{k,1}$ rectifiable if and only if $\mu$ is $C^{k+1}$ rectifiable.\end{enumerate}
\end{theorem}

\begin{proof} For (1), see \cite[Proposition 3.3]{AS94}. When $\mu=\Haus^m\res E$ and $0<\Haus^m(E)<\infty$, (2) is a consequence of \cite[3.1.5]{Federer}. The assertion for absolutely continuous measures then follows via the Radon-Nikodym theorem. Compare to Theorem \ref{t:equiv} above. \end{proof}

Assume that $M\subset\RR^n$ with $\Haus^m(M)<\infty$. Let $x\in M$ and assume that $M$ admits a unique approximate tangent plane $T_x M$ at $x$. Following \cite{AS94}, we say that a closed set $Q_x M\subset\RR^n$ is the \emph{approximate tangent paraboloid of $M$ at $x$} if \begin{itemize}
\item $Q_xM=\{(y,A(y,y)):y\in T_x M\}$ for some bilinear symmetric form $$A:T_xM \times T_x M \rightarrow (T_xM)^\perp,$$
\item $\Haus^m\res \xi_{x,\rho}(M)$ converges to $\Haus^m\res Q_x M$ weakly in the sense of Radon measures as $\rho\downarrow 0$, where $\xi_{x,\rho}$ denotes the \emph{nonhomogeneous dilation of $M$ at $x$} defined by $$\xi_{x,\rho}:\RR^n\rightarrow\RR^n,\quad \xi_{x,\rho}(y)= \rho^{-1}\pi_{T_xM}(y-x) + \rho^{-2} \pi_{T_xM^\perp}(y-x).$$
    \end{itemize}

\begin{example}[approximate tangent paraboloid to a graph] Let $f:\RR^m\rightarrow\RR^{n-m}$ be a function of class $C^2$ such that $f(0)=0$, $Df(0)=0$, and $$f(x)=\frac{1}{2}\sum_{p,q=1}^n D_{p,q}^2f(0)x_px_q + o(|x|^2)=:g(x)+o(|x|^2).$$ Then the graph of $f$ is a $C^2$ manifold and $Q_{0}(\mathrm{graph}\,f) = \{(x,g(x)):x\in\RR^m\}$.\end{example}

Anzellotti and Serapioni \cite{AS94} give geometric characterizations of $C^{1,\alpha}$ rectifiability when $\alpha\in[0,1]$. In the case $\alpha=1$, their result is the following.

\begin{theorem}[see {\cite[Theorem 3.5]{AS94}}; geometric characterization of $C^2$ rectifiability] Assume $M\subset\RR^n$ with $\Haus^m(M)<\infty$. Then $\Haus^m\res M$ is $C^2$ $m$-rectifiable if and only if $\Haus^m(M\setminus \bigcup_{j=1}^\infty M_j)=0$ for some $M_j\subset\RR^n$ with $M_i\cap M_j=\emptyset$ for all $i\neq j$ and such that for every $j\geq 1$ and $\Haus^m$-a.e. $x\in M_j$, \begin{enumerate}
\item the approximate tangent plane $T_x M_j$ exists,
\item the approximate tangent paraboloid $Q_x M_j$ exists, and
\item (see \cite[2.9.12]{Federer} for definition of approximate limits) $$\mathrm{ap}\,\limsup_{\stackrel{y\rightarrow x}{y\in M_j}} \dfrac{d(T_yM_j,T_xM_j)}{|y-x|}<\infty,\quad d(T_1,T_2)\equiv\sup_{|y|=1} |\pi_{T_1}(y)-\pi_{T_2}(y)|.$$\end{enumerate}
\end{theorem}

More recently, two new characterizations of $C^{1,\alpha}$ rectifiability of absolutely continuous measures have been provided by Kolasi\'nski \cite{Kol} (using Menger-type curvatures) and by Ghinassi \cite{Ghinassi} (using $L^2$ beta numbers), the latter of which can be stated as follows. See \eqref{eq:beta-m} for the definition of $\beta_2^{m,h}(\mu,x,r)$.

\begin{theorem}[see {\cite[Theorem II]{Ghinassi}}] Let $\mu$ be a Radon measure on $\RR^n$ with $\mu\ll\Haus^m$. If $\lD{m}(\mu,x)>0$ at $\mu$-a.e.~$x\in\RR^n$ and \begin{equation}\label{e:silvia}\int_0^1 \frac{\beta_2^{m,h}(\mu,x,r)^2}{r^{2\alpha}} \frac{dr}{r}<\infty\quad\text{at $\mu$-a.e. }x\in\RR^n,\end{equation} then $\mu$ is $C^{1,\alpha}$ $m$-rectifiable.\end{theorem}

It would be interesting to know whether Ghinassi's sufficient condition is also necessary for $C^{1,\alpha}$ rectifiability (compare to Theorem \ref{t:assort}(6)). The following problem is also open.

\begin{problem}[identification problem for $C^{k,\alpha}$ $m$-rectifiable measures] \label{p:higher} Find geometric or measure-theoretic properties that characterize $C^{k,\alpha}$ $m$-rectifiable measures when $k\geq 2$. (To start, you may assume $\mu\ll\Haus^m$.) \end{problem}

For recent progress on Problem \ref{p:higher} for Hausdorff measures, see Santilli \cite{Santilli}.

\subsection{Rectifiability in other spaces} Following Kirchheim \cite{Kirchheim}, a metric space $(X,\rho)$ is called \emph{$m$-rectifiable} if the $m$-dimensional Hausdorff measure on $X$ is carried by images of Lipschitz maps from subsets of $\RR^m$ into $X$. When $E\subset \RR^n$, the space $(E,d_2)$ equipped with induced Euclidean metric is $m$-rectifiable if and only if $\Haus^m\res E$ is Lipschitz image rectifiable in the sense of \S3. Kirchheim examined the structure of general $m$-rectifiable metric spaces and proved that for those spaces with $\Haus^m(X)<\infty$, the $m$-dimensional Hausdorff density exists and equals 1 at $\Haus^m$ almost every point.


\begin{theorem}[see \cite{Kirchheim}] \label{t:kirk} Let $(X,\rho)$ be a metric space and assume that $\Haus^m(X)<\infty$ for some integer $m\geq 1$. If $X$ is $m$-rectifiable, then \begin{equation}\label{e:d1} \lim_{r\downarrow 0} \frac{\Haus^m(B_\rho(x,r))}{\omega_m r^m}=1\quad\text{at $\Haus^m$-a.e. }x\in X.\end{equation}\end{theorem}

When $(X,\rho)=(E,d_2)$, where $E\subset\RR^n$ is equipped with the Euclidean metric and $\Haus^m(E)<\infty$, the converse of Theorem \ref{t:kirk} is also true. This was proved by Besicovitch \cite{Bes28} when $m=1$ and $n=2$ and by Mattila \cite{Mattila75} for general $m$ and $n$. The converse of Theorem \ref{t:kirk} is also true when $m=1$ for general metric spaces with $\Haus^1(X)<\infty$  by Preiss and Ti\v{s}er \cite{PT92} (recall Example \ref{ex:bes} above). In all other cases it is not presently known whether the converse of Kirchheim's theorem is true or false.

\begin{problem}\label{p:kirk} Let $m\geq 2$. Prove that for every metric space $(X,\rho)$ with $\Haus^m(X)<\infty$ that \eqref{e:d1} implies $X$ is $m$-rectifiable. Or find a counterexample.\end{problem}

Problem \ref{p:kirk} is interesting even in the case when $(X,\rho)=(E,d_p)$ for some $E\subset\RR^n$ and $d_p$ is the distance induced by the $p$-norm, $p\neq 2$. For related work on existence of densities of measures in Euclidean spaces with respect to non-spherical norms, see the series of papers by Lorent \cite{Lorent03,Lorent04,Lorent07}.

Although a density only characterization of rectifiable metric spaces remains illusive, a metric analysis characterization of rectifiable spaces has recently been established by Bate and Li \cite{Bate-Li}. \emph{Lipschitz differentiability spaces} were introduced by Cheeger \cite{Cheeger} and examined in depth by Bate \cite{Bate}. Roughly speaking, these are spaces that have a sufficiently rich curve structure to support a version of Rademacher's theorem; we refer the reader to \cite{Bate} for a detailed description and several characterizations of differentiability spaces. The following theorem is a simplified statement of Bate and Li's main result.

\begin{theorem}[see \cite{Bate-Li}] A metric space $(X,\rho)$ is $m$-rectifiable if and only if $\Haus^m\res X$ is carried by Borel sets $U\subset X$ such that \begin{itemize}
\item $(U,\rho,\Haus^m)$ is an $m$-dimensional Lipschitz differentiability space, and \item $0<\lD{m}(\Haus^m\res U,x)\leq \uD{m}(\Haus^m\res U,x)<\infty$ at $\Haus^m$-a.e. $x\in U$.\end{itemize}\end{theorem}

A metric space $(X,\rho)$ is called \emph{purely $m$-rectifiable} if $\Haus^m\res X$ is singular to images of Lipschitz maps from subsets of $\RR^m$ into $X$. For example, the (first) Heisenberg group $\mathbb{H}$ with topological dimension 3 and Hausdorff dimension 4 is purely $m$-unrectifiable for all $m=2,3,4$ (see Ambrosio and Kirchheim \cite[\S7]{AK}). A notion of intrinsic rectifiability of sets in Heisenberg groups (i.e.~ rectifiability with respect to $C^1$ images of homogeneous subgroups) was investigated by Mattila, Serapioni, and Serra Cassano \cite{MSSC}. For related developments, see \cite{CT15} and \cite{CFO}.
A characterization of complete, purely $m$-unrectifiable metric spaces with $\Haus^m(X)<\infty$ was recently announced by Bate \cite{Bate-pu}. A related sufficient condition for rectifiability of a metric space (using Bate's theorem) has been announced by David and Le Donne \cite{DL-rect}.

\bibliography{ident}{}
\bibliographystyle{amsbeta}

\end{document}